\documentclass[12pt]{article}
\usepackage{amssymb,amsthm,hyperref}
\newtheorem{thm}{Theorem}[section]
 
 \newtheorem{cor}{Corollary}[section]
 \newtheorem{lem}{Lemma}[section]
 \newtheorem{prop}{Proposition}[section]
 \newtheorem{defn}{Definition}[section]
\theoremstyle{remark}
\newtheorem{rem}{Remark}[section]

\title{Relaxation limit in larger Besov spaces for compressible Euler equations}
\author{Jiang Xu\thanks {E-mail: jiangxu\underline{ }79@yahoo.com.cn}\\
\small{\textit{Department of Mathematics}},\\\small{\textit{Nanjing
University of Aeronautics and Astronautics}},
\\ \small{\textit{Nanjing 211106, P.R.China}}\\[5mm]
Shuichi Kawashima\thanks{E-mail: kawashim@math.kyushu-u.ac.jp}\\
\small{\textit{Graduate School of Mathematics}},\\
\small{\textit{Kyushu University, Fukuoka 812-8581, Japan}}}
\date{}
\begin{document}
\maketitle{} \begin{abstract}The work is devoted to the relaxation
limit in larger Besov spaces for compressible Euler equations, which
contains previous results in Sobolev spaces and Besov spaces with
critical regularity. Such an extension depends on a revision of
commutator estimates and an elementary fact which indicates the
connection between homogeneous and inhomogeneous Chemin-Lerner
spaces.
\end{abstract}

\hspace{-0.5cm}\textbf{Keywords.} \small{compressible Euler
equations,
 relaxation limit, Chemin-Lerner spaces}\\

\hspace{-0.5cm}\textbf{AMS subject classification:} \small{35L25,\
35L45,\ 76N15}

\section{Introduction}
In this paper, we consider the following nondimensional compressible
Euler equations
\begin{equation}
\left\{
\begin{array}{l}
\partial_{t}\rho + \nabla\cdot(\rho\textbf{v}) = 0 , \\
\partial_{t}(\rho\textbf{v}) +\nabla\cdot(\rho\textbf{v}\otimes\textbf{v}) +
\nabla p(\rho) =-\rho\textbf{v}/\tau
\end{array} \right.\label{R-E1}
\end{equation}
for $(t,x)\in[0,+\infty)\times\mathbb{R}^{N}$ with $N\geq1$. Here
 $\rho = \rho(t, x)$ is the fluid density function;
$\textbf{v}=(v^1,v^2,\cdot\cdot\cdot,v^{d})^{\top}$($\top$
represents the transpose) denotes the fluid velocity. The pressure
$p(\rho)$ satisfies the classical assumption
$$p'(\rho)>0, \ \ \ \mbox{for any}\ \ \rho>0.$$
An usual simplicity $p(\rho):=\rho^{\gamma}(\gamma\geq 1)$, where
the adiabatic exponent $\gamma>1$ corresponds to the isentropic flow
and $\gamma=1$ corresponds to the isothermal flow. The
nondimensional number $0<\tau\leq1$ is a (small) relaxation time.
The notation $\nabla,\otimes$ are the gradient operator (in $x$) and
the symbol for the tensor products of two vectors, respectively.

System (\ref{R-E1}) is complemented by the initial conditions
\begin{equation}
(\rho,\textbf{v})(0,x)=(\rho_{0},\textbf{v}_{0}).\label{R-E2}
\end{equation}

For fixed $\tau>0$, as we all know, the relaxation term which plays
the role of damping can prevent the finite time blow-up and the
Cauchy problem (\ref{R-E1})-(\ref{R-E2}) admits a unique global
classical solution, provided the initial data is small under certain
norms. In this direction, such problem was widely studied by many
authors, see e.g. \cite{HL,HP2,HMP,LW,N2,NY,STW,WY} and references
therein. In addition, it is proved that the solutions in \cite{STW}
has the $L^{\infty}$ convergence rate $(1+t)^{-3/2}(N=3)$ to the
constant background state and the optimal $L^{p}(1<p\leq\infty)$
convergence rate $(1+t)^{-N/2(1-1/p)}$ in general several dimensions
\cite{WY}, respectively. In one space dimension in Lagrangian
coordinates, Nishida \cite{N2} obtained the global classical
solutions with small data, and the solutions following Darcy's law
asymptotically as time tends to infinity was shown by Hsiao and Liu
\cite{HL}. For the large-time behavior of solutions with vacuum, the
reader is referred to \cite{HP2,HMP}. Nishihara and Yang \cite{NY}
studied the boundary effect on the asymptotic behavior of the
solutions to the one-dimensional initial-boundary value problem.
Later, Liu and Wang \cite{LW} considered the 2-D initial-boundary
value problem in the wedge-space, and proved the global existence of
classical solutions.

Another interesting line of research is to justify the singular
limit as $\tau\rightarrow0$ in (\ref{R-E1}).  To do this, we
introduce the time variable by considering an
``$\mathcal{O}(1/\tau)$" time scale
\begin{equation}(\rho^{\tau},\textbf{v}^{\tau})(s,x)=\Big(\rho,\textbf{v}\Big)\Big(\frac{s}{\tau},x\Big). \label{R-E3} \end{equation}
Then
\begin{equation}
\left\{
\begin{array}{l}\partial_{s}\rho^{\tau}+\nabla\cdot(\frac{\rho^{\tau}\textbf{v}^{\tau}}{\tau})=0,\cr
 \tau^2\partial_{s}(\frac{\rho^{\tau}\textbf{v}^{\tau}}{\tau})+\tau^2\nabla\cdot(\frac{\rho^{\tau}\textbf{v}^{\tau}\otimes\textbf{v}^{\tau}}{\tau^2})+\frac{\rho^{\tau}\textbf{v}^{\tau}}{\tau}=-\nabla
 p(\rho^{\tau})
\end{array} \right.\label{R-E4}
\end{equation}
with the initial data
\begin{equation}(\rho^{\tau},\textbf{v}^{\tau})(x,0)=(\rho_{0},
\textbf{v}_{0}).\label{R-E5}\end{equation} At the formal level, at
least, if we can assume that $\frac{\rho^\tau
\mathbf{v}^\tau}{\tau}$ is uniformly bounded, it will be shown that
the limit $\mathcal{N}$ of $\rho^{\tau}$ as $\tau\rightarrow0$
satisfies the classical porous medium equation
\begin{equation} \left\{
\begin{array}{l}\partial_{s}\mathcal{N}-\Delta p(\mathcal{N})=0,\\
\mathcal{N}(x,0)=\rho_{0},
\end{array} \right.\label{R-E6}
\end{equation}
which is a parabolic equation since $p(\mathcal{N})$ is strictly
increasing.

This singular limit from hyperbolic relaxation to parabolic
equations have attracted much attention, see
\cite{CG,LC,JR,MM,MMS,MR,X} and therein references. The relaxation
results mentioned for smooth solutions fell in the framework of the
classical existence theory of Kato and Majda\cite{K,M}. The
regularity index of Sobolev spaces (in $x$)
$H^{\sigma}(\mathbb{R}^{N})$ is assumed to be high ($\sigma>1+N/2$
with \textit{integer}). Recently, the first author and Wang
\cite{XW} studied the limit case of regularity index
($\sigma=1+N/2$) where the classical theory fails. They developed a
new commutator estimate to overcome the technique difficulty and
constructed global classical solutions in the critical Besov spaces
$B^{1+N/2}_{2,1}(\mathbb{R}^{N})$. Furthermore, based on the
Aubin-Lions compactness lemma in \cite{S}, it was justified that the
(scaled) density converges to the solution of the porous medium
equation.

The main purpose of this paper is to generalize the relaxation
limits in Sobolev spaces with higher regularity and Besov spaces
with critical regularity. Due to the elementary fact that Sobolev
spaces $H^{\sigma}(\mathbb{R}^{N}):=B^\sigma_{2,2}(\mathbb{R}^{N})$,
a natural question is whether those results hold in larger Besov
spaces $B^{\sigma}_{2,r}(\mathbb{R}^{N})$ or not, whose indices
satisfy the following condition:
\begin{eqnarray} \sigma>1+N/2, \ 1\leq r\leq2\ \ \mbox{or}\ \ \sigma=1+N/2,\  r=1. \label{R-E666}\end{eqnarray}
In the present paper, we shall answer the question. The main
difficulty lies in the a priori nonlinear estimates, in particular,
the commutator estimates. To overcome it, we need a more general
version of commutator estimates in Proposition \ref{prop2.3} (also
see \cite{D}), which relaxes the restriction on the couple $(s,r)$.
For completeness and the reader convenience, we present the proof in
Appendix with the aid of the Bony's decomposition. Besides, this
extension also heavily depends on an elementary fact developed for
general hyperbolic system of balance laws (see Lemma \ref{lem4.1} or
\cite{XK}), which indicates the connection between homogeneous and
inhomogeneous Chemin-Lerner spaces. Precisely, our results are
stated as follows.

\begin{thm}\label{thm1.1}
Let the couple $(\sigma, r)$ satisfy the condition (\ref{R-E666})
and let $\bar{\rho}>0$ be a constant reference density. There exists
a positive constant $\delta_{0}$ independent of $\tau$ such that if
$$\|(\rho_{0}-\bar{\rho},\mathbf{m}_{0})\|_{B^{\sigma}_{2,r}(\mathbb{R}^{N})}\leq
\delta_{0}$$ with $\mathbf{m}_{0}:=\rho_{0}\mathbf{v}_{0}$, then the
Cauchy problem (\ref{R-E1})-(\ref{R-E2}) has a unique global
classical solution $(\rho,\mathbf{m})\in
\mathcal{C}^{1}(\mathbb{R}^{+}\times \mathbb{R}^{N})$ satisfying $
(\rho-\bar{\rho},\mathbf{m}) \in
\widetilde{\mathcal{C}}(B^{\sigma}_{2,r}(\mathbb{R}^{N}))\cap
\widetilde{\mathcal{C}}^1(B^{\sigma-1}_{2,r}(\mathbb{R}^{N})). $
Furthermore, the global solutions satisfy the inequality
\begin{eqnarray}
&&\|(\rho-\bar{\rho},\mathbf{m})\|_{\widetilde{L}^\infty(B^{\sigma}_{2,r}(\mathbb{R}^{N}))}
+\mu_{0}\Big(\Big\|\frac{\mathbf{m}}{\sqrt{\tau}}\Big\|_{\widetilde{L}^2(B^{\sigma}_{2,r}(\mathbb{R}^{N}))}
+\Big\|\sqrt{\tau}(\nabla\rho,\nabla\mathbf{m})\Big\|_{\widetilde{L}^2(B^{\sigma-1}_{2,r}(\mathbb{R}^{N}))}\Big)
\nonumber\\&\leq& C_{0}\|(\rho_{0}-\bar{\rho},
\mathbf{m}_{0})\|_{B^{\sigma}_{2,r}(\mathbb{R}^{N})}, \label{R-E7}
\end{eqnarray}
where $\mathbf{m}:=\rho\mathbf{v}$, $\mu_{0}$ and $C_{0}$ are some
uniform positive constants independent of $\tau (0<\tau\leq1)$.
\end{thm}
\begin{rem}
It should be pointed out Theorem \ref{thm1.1} contains previous
results (see \cite{CG,LC,STW,WY,XW} and the references therein) for
compressible Euler equations. The proof depends on the Fourier
localization methods and Shizuta-Kawashima algebraic condition which
has been well developed by the second author \textit{et al.}
\cite{KY,Y} for generally hyperbolic systems of balance laws. In
particular, a concrete entropy for current compressible Euler
equations is available, which has been verified in, \textit{e.g.},
\cite{LC,XK}. In addition,  we track the singular parameter $\tau$
in the uniform energy inequality (\ref{R-E7}), which plays a key
role in the analysis of relaxation limit problem.
\end{rem}

As a direct consequence, using the standard weak convergence method
and Aubin-Lions compactness lemma (\cite{S}), we further have the
analogue relaxation convergence as in \cite{XW}.

\begin{thm}\label{thm1.2} Let $(\rho,\mathbf{m})$ be the global solution of Theorem
\ref{thm1.1}. Then it yields
$$\rho^{\tau}-\bar{\rho}\ \ \ \mbox{is uniformly bounded in}\ \ \mathcal{C}(\mathbb{R}^{+},B^{\sigma}_{2,r}(\mathbb{R}^{N}));$$
$$\frac{\rho^{\tau}\mathbf{v}^{\tau}}{\tau}\ \ \ \mbox{is uniformly bounded in}\ \ L^2(\mathbb{R}^{+},B^{\sigma}_{2,r}(\mathbb{R}^{N})).$$
Furthermore, there exists some function $\mathcal{N}\in
\mathcal{C}(\mathbb{R}^{+},
\bar{n}+B^{\sigma}_{2,r}(\mathbb{R}^{N}))$ which is a unique
solution of (\ref{R-E6}). For any $0<T,R<\infty$,
$\{\rho^{\tau}(s,x)\}$ strongly converges to $\mathcal{N}(s,x)$ in
 $\mathcal{C}([0,T],
(B^{\sigma-\delta}_{2,r}(B_{r}))$ as  $\tau\rightarrow0$, where
$\delta\in(0,1)$ and $B_{r}$ denotes the ball of radius $r$ in
$\mathbb{R}^{N}$. In addition, it holds that
\begin{eqnarray}
\|(\mathcal{N}(s,\cdot)-\bar{\rho}\|_{B^{\sigma}_{2,r}(\mathbb{R}^{N})}\leq
C_{0}\|(\rho_{0}-\bar{\rho},
\mathbf{m}_{0})\|_{B^{\sigma}_{2,r}(\mathbb{R}^{N})},\ s\geq0,
\label{R-E1000}
\end{eqnarray}
where $C_{0}>0$ is a uniform constant independent of $\tau$.
\end{thm}

\begin{rem}
Theorem \ref{thm1.2} gives a rigorous description in \textit{larger}
spaces that the porous medium equation is usually regarded as an
appropriate model for compressible inviscid fluids in small
amplitude regime of relaxation time $\tau$. In comparison with our
recent results in \cite{XW}, Theorem \ref{thm1.1}-\ref{thm1.2} also
hold in the cases of \textit{general} pressure and
\textit{arbitrary} space dimensions except for the regularity
consideration.
\end{rem}

\begin{rem}
From the symmetrization in Sect.\ref{sec:3}, we see that the
dependence of the matrices $A^{j}\ (j=0,1,2\cdot\cdot\cdot,N)$ with
respect to the total variable $W$ rather than only $W_{2}$. The
concrete context of matrices enables us to obtain the uniform
frequency-localization estimates and the relaxation limit. However,
to the best of our knowledge, it is unknown for generally hyperbolic
systems to get corresponding results. Therefore, this paper can be
regarded as the effort to the open question in \cite{BZ} (Remark 15,
p.225).
\end{rem}

The rest of this paper unfolds as follows. In Sect.~\ref{sec:2}, we
briefly review the Littlewood-Paley decomposition and properties of
Besov spaces and Chemin-Lerner spaces. In Sect.~\ref{sec:3}, we
reformulate the equations (\ref{R-E1}) as a symmetric hyperbolic
form in terms of entropy variables, and the local existence and
blow-up criterion of classical solutions in critical spaces are
presented. Sect.~\ref{sec:4} is devoted to deduce the a priori
estimates in Chemin-Lerner spaces by using Fourier-localization
arguments, which are used to achieve the global existence of
classical solutions.

\section{Preliminary}\label{sec:2}
\setcounter{equation}{0} Throughout the paper, $f\lesssim g$ denotes
$f\leq Cg$, where $C>0$ is a generic constant. $f\thickapprox g$
means $f\lesssim g$ and $g\lesssim f$. Denote by
$\mathcal{C}([0,T],X)$ (resp., $\mathcal{C}^{1}([0,T],X)$) the space
of continuous (resp., continuously differentiable) functions on
$[0,T]$ with values in a Banach space $X$.  Also, $\|(f,g,h)\|_{X}$
means $ \|f\|_{X}+\|g\|_{X}$, where $f,g\in X$. $\langle f,g\rangle$
denotes the inner product of two functions $f,g$ in
$L^2(\mathbb{R}^{N})$.

In this section, we briefly review the Littlewood-Paley
decomposition and some properties of Besov spaces. The reader is
also referred to, \textit{e.g.}, \cite{BCD} for more details. We
omit the space dependence, since all functional spaces (in $x$) are
considered in $\mathbb{R}^{N}$.

Let us start with the Fourier transform. The Fourier transform
$\hat{f}$ of a $L^1$-function $f$ is given by
$$\mathcal{F}f=\int_{\mathbb{R}^{N}}f(x)e^{-2\pi x\cdot\xi}dx.$$ More
generally, the Fourier transform of any $f\in\mathcal{S}'$, the
space of tempered distributions, is given by
$$(\mathcal{F}f,g)=(f,\mathcal{F}g)$$ for any $g\in \mathcal{S}$, the Schwartz
class.

First, we fix some notation.
$$\mathcal{S}_{0}=\Big\{\phi\in\mathcal{S},\partial^{\alpha}\mathcal{F}f(0)=0, \forall \alpha \in \mathbb{N}^{N}\ \mbox{multi-index}\Big\}.$$
Its dual is given by
$$\mathcal{S}'_{0}=\mathcal{S}'/\mathcal{P},$$ where $\mathcal{P}$
is the space of polynomials.

We now introduce a dyadic partition of $\mathbb{R}^{N}$. We choose
$\phi_{0}\in \mathcal{S}$ such that $\phi_{0}$ is even,
$$\mathrm{supp}\phi_{0}:=A_{0}=\Big\{\xi\in\mathbb{R}^{N}:\frac{3}{4}\leq|\xi|\leq\frac{8}{3}\Big\},\  \mbox{and}\ \ \phi_{0}>0\ \ \mbox{on}\ \ A_{0}.$$
Set $A_{q}=2^{q}A_{0}$ for $q\in\mathbb{Z}$. Furthermore, we define
$$\phi_{q}(\xi)=\phi_{0}(2^{-q}\xi)$$ and define $\Phi_{q}\in
\mathcal{S}$ by
$$\mathcal{F}\Phi_{q}(\xi)=\frac{\phi_{q}(\xi)}{\sum_{q\in \mathbb{Z}}\phi_{q}(\xi)}.$$
It follows that both $\mathcal{F}\Phi_{q}(\xi)$ and $\Phi_{q}$ are
even and satisfy the following properties:
$$\mathcal{F}\Phi_{q}(\xi)=\mathcal{F}\Phi_{0}(2^{-q}\xi),\ \ \ \mathrm{supp}\ \mathcal{F}\Phi_{q}(\xi)\subset A_{q},\ \ \ \Phi_{q}(x)=2^{qN}\Phi_{0}(2^{q}x)$$
and
$$\sum_{q=-\infty}^{\infty}\mathcal{F}\Phi_{q}(\xi)=\cases{1,\ \ \ \mbox{if}\ \ \xi\in\mathbb{R}^{N}\setminus \{0\},
\cr 0, \ \ \ \mbox{if}\ \ \xi=0.}
$$
As a consequence, for any $f\in S'_{0},$ we have
$$\sum_{q=-\infty}^{\infty}\Phi_{q}\ast f=f.$$

To define the homogeneous Besov spaces, we set
$$\dot{\Delta}_{q}f=\Phi_{q}\ast f,\ \ \ \ q=0,\pm1,\pm2,...$$

\begin{defn}\label{defn2.1}
For $s\in \mathbb{R}$ and $1\leq p,r\leq\infty,$ the homogeneous
Besov spaces $\dot{B}^{s}_{p,r}$ is defined by
$$\dot{B}^{s}_{p,r}=\{f\in S'_{0}:\|f\|_{\dot{B}^{s}_{p,r}}<\infty\},$$
where
$$\|f\|_{\dot{B}^{s}_{p,r}}
=\cases{\Big(\sum_{q\in\mathbb{Z}}(2^{qs}\|\dot{\Delta}_{q}f\|_{L^p})^{r}\Big)^{1/r},\
\ r<\infty, \cr \sup_{q\in\mathbb{Z}}
2^{qs}\|\dot{\Delta}_{q}f\|_{L^p},\ \ r=\infty.} $$\end{defn}

To define the inhomogeneous Besov spaces, we set $\Psi\in
\mathcal{C}_{0}^{\infty}(\mathbb{R}^{N})$ be even and satisfy
$$\mathcal{F}\Psi(\xi)=1-\sum_{q=0}^{\infty}\mathcal{F}\Phi_{q}(\xi).$$
It is clear that for any $f\in S'$, yields
$$\Psi*f+\sum_{q=0}^{\infty}\Phi_{q}\ast f=f.$$
We further set
$$\Delta_{q}f=\cases{0,\ \ \ \ \ \ \ \, \ q\leq-2,\cr
\Psi*f,\ \ \ q=-1,\cr \Phi_{q}\ast f, \ \ q=0,1,2,...}$$

\begin{defn}\label{defn2.2}
For $s\in \mathbb{R}$ and $1\leq p,r\leq\infty,$ the inhomogeneous
Besov spaces $B^{s}_{p,r}$ is defined by
$$B^{s}_{p,r}=\{f\in S':\|f\|_{B^{s}_{p,r}}<\infty\},$$
where
$$\|f\|_{B^{s}_{p,r}}
=\cases{\Big(\sum_{q=-1}^{\infty}(2^{qs}\|\Delta_{q}f\|_{L^p})^{r}\Big)^{1/r},\
\ r<\infty, \cr \sup_{q\geq-1} 2^{qs}\|\Delta_{q}f\|_{L^p},\ \
r=\infty.}$$
\end{defn}
Next we turn to Bernstein inequalities.
\begin{lem}\label{lem2.1}
Let $k\in\mathbb{N}$ and $0<R_{1}<R_{2}$. There exists a constant
$C$, depending only on $R_{1},R_{2}$ and $N$, such that for all
$1\leq a\leq b\leq\infty$ and $f\in L^{a}$,
$$
\mathrm{Supp}\mathcal{F}f\subset \{\xi\in \mathbb{R}^{N}: |\xi|\leq
R_{1}\lambda\}\Rightarrow\sup_{|\alpha|=k}\|\partial^{\alpha}f\|_{L^{b}}
\leq C^{k+1}\lambda^{k+d(\frac{1}{a}-\frac{1}{b})}\|f\|_{L^{a}};
$$
$$
\mathrm{Supp}\mathcal{F}f\subset \{\xi\in \mathbb{R}^{N}:
R_{1}\lambda\leq|\xi|\leq R_{2}\lambda\} \Rightarrow
C^{-k-1}\lambda^{k}\|f\|_{L^{a}}\leq
\sup_{|\alpha|=k}\|\partial^{\alpha}f\|_{L^{a}}\leq
C^{k+1}\lambda^{k}\|f\|_{L^{a}}.
$$
\end{lem}

As a direct corollary of the above inequalities, it holds that
\begin{rem}\label{rem2.1}
$$\frac{1}{C}\|f\|_{\dot{B}^{s + |\alpha|}_{p,
r}}\leq\|\partial^\alpha f\|_{\dot{B}^s_{p, r}}\leq
C\|f\|_{\dot{B}^{s + |\alpha|}_{p, r}};$$
$$
\|\partial^\alpha f\|_{B^s_{p, r}}\leq C\|f\|_{B^{s + |\alpha|}_{p,
r}},
$$ for all
multi-index $\alpha$.
\end{rem}

The Besov spaces defined above obey various inclusion relations. In
particular, we have
\begin{lem}\label{lem2.2} Let $s\in \mathbb{R}$ and $1\leq
p,r\leq\infty,$ then
\begin{itemize}
\item[(1)]If $s>0$, then $B^{s}_{p,r}=L^{p}\cap \dot{B}^{s}_{p,r};$
\item[(2)]If $\tilde{s}\leq s$, then $B^{s}_{p,r}\hookrightarrow
B^{\tilde{s}}_{p,\tilde{r}}$;
\item[(3)]$\dot{B}^{N/p}_{p,1}\hookrightarrow\mathcal{C}_{0},\ \ B^{N/p}_{p,1}\hookrightarrow\mathcal{C}_{0}(1\leq p<\infty);$
\end{itemize}
where $\mathcal{C}_{0}$ is the space of continuous bounded functions
which decay at infinity.
\end{lem}

On the other hand, we also present the definition of Chemin-Lerner
space-time spaces initialed by J.-Y. Chemin and N. Lerner \cite{C2},
which are the refinement of the spaces
$L^{\theta}_{T}(\dot{B}^{s}_{p,r})$ or
$L^{\theta}_{T}(B^{s}_{p,r})$.

\begin{defn}\label{defn2.3}
For $T>0, s\in\mathbb{R}, 1\leq r,\theta\leq\infty$, the homogeneous
mixed time-space Besov spaces
$\widetilde{L}^{\theta}_{T}(\dot{B}^{s}_{p,r})$ is defined by
$$\widetilde{L}^{\theta}_{T}(\dot{B}^{s}_{p,r}):
=\{f\in
L^{\theta}(0,T;\mathcal{S}'_{0}):\|f\|_{\widetilde{L}^{\theta}_{T}(\dot{B}^{s}_{p,r})}<+\infty\},$$
where
$$\|f\|_{\widetilde{L}^{\theta}_{T}(\dot{B}^{s}_{p,r})}:=\Big(\sum_{q\in\mathbb{Z}}(2^{qs}\|\dot{\Delta}_{q}f\|_{L^{\theta}_{T}(L^{p})})^{r}\Big)^{\frac{1}{r}}$$
with the usual convention if $r=\infty$.
\end{defn}

\begin{defn}\label{defn2.4}
For $T>0, s\in\mathbb{R}, 1\leq r,\theta\leq\infty$, the
inhomogeneous mixed time-space Besov spaces
$\widetilde{L}^{\theta}_{T}(B^{s}_{p,r})$ is defined by
$$\widetilde{L}^{\theta}_{T}(B^{s}_{p,r}):
=\{f\in
L^{\theta}(0,T;\mathcal{S}'):\|f\|_{\widetilde{L}^{\theta}_{T}(B^{s}_{p,r})}<+\infty\},$$
where
$$\|f\|_{\widetilde{L}^{\theta}_{T}(B^{s}_{p,r})}:=\Big(\sum_{q\geq-1}(2^{qs}\|\Delta_{q}f\|_{L^{\theta}_{T}(L^{p})})^{r}\Big)^{\frac{1}{r}}$$
with the usual convention if $r=\infty$.
\end{defn}

We only state some basic properties on the inhomogeneous
Chemin-Lerner spaces, since the similar ones follow in the
homogeneous Chemin-Lerner spaces.

The first one is that $\widetilde{L}^{\theta}_{T}(B^{s}_{p,r})$ may
be linked with the classical spaces $L^{\theta}_{T}(B^{s}_{p,r})$
via the Minkowski's inequality:
\begin{rem}\label{rem2.2}
It holds that
$$\|f\|_{\widetilde{L}^{\theta}_{T}(B^{s}_{p,r})}\leq\|f\|_{L^{\theta}_{T}(B^{s}_{p,r})}\,\,\,
\mbox{if}\,\, r\geq\theta;\ \ \ \
\|f\|_{\widetilde{L}^{\theta}_{T}(B^{s}_{p,r})}\geq\|f\|_{L^{\theta}_{T}(B^{s}_{p,r})}\,\,\,
\mbox{if}\,\, r\leq\theta.
$$\end{rem}
Let us also recall the property of continuity for product in
Chemin-Lerner spaces $\widetilde{L}^{\theta}_{T}(B^{s}_{p,r})$.
\begin{prop}\label{prop2.1}
The following inequality holds:
$$
\|fg\|_{\widetilde{L}^{\theta}_{T}(B^{s}_{p,r})}\leq
C(\|f\|_{L^{\theta_{1}}_{T}(L^{\infty})}\|g\|_{\widetilde{L}^{\theta_{2}}_{T}(B^{s}_{p,r})}
+\|g\|_{L^{\theta_{3}}_{T}(L^{\infty})}\|f\|_{\widetilde{L}^{\theta_{4}}_{T}(B^{s}_{p,r})})
$$
whenever $s>0, 1\leq p\leq\infty,
1\leq\theta,\theta_{1},\theta_{2},\theta_{3},\theta_{4}\leq\infty$
and
$$\frac{1}{\theta}=\frac{1}{\theta_{1}}+\frac{1}{\theta_{2}}=\frac{1}{\theta_{3}}+\frac{1}{\theta_{4}}.$$
As a direct corollary, one has
$$\|fg\|_{\widetilde{L}^{\theta}_{T}(B^{s}_{p,r})}
\leq
C\|f\|_{\widetilde{L}^{\theta_{1}}_{T}(B^{s}_{p,r})}\|g\|_{\widetilde{L}^{\theta_{2}}_{T}(B^{s}_{p,r})}$$
whenever $s\geq d/p,
\frac{1}{\theta}=\frac{1}{\theta_{1}}+\frac{1}{\theta_{2}}.$
\end{prop}

In the next symmetrization, we meet with some composition functions.
The following continuity result for compositions is used to estimate
them.
\begin{prop}\label{prop2.2}
Let $s>0$, $1\leq p, r, \theta\leq \infty$, $F'\in
W^{[s]+1,\infty}_{loc}(I;\mathbb{R})$ with $F(0)=0$, $T\in
(0,\infty]$ and $v\in \widetilde{L}^{\theta}_{T}(B^{s}_{p,r})\cap
L^{\infty}_{T}(L^{\infty}).$ Then
$$\|F(f)\|_{\widetilde{L}^{\theta}_{T}(B^{s}_{p,r})}\leq
C(1+\|f\|_{L^{\infty}_{T}(L^{\infty})})^{[s]+1}\|F'\|_{W^{[s]+1,\infty}}\|f\|_{\widetilde{L}^{\theta}_{T}(B^{s}_{p,r})}.$$
\end{prop}

In addition, we present some estimates of commutators in homogeneous
and inhomogeneous Chemin-Lerner spaces to bound commutators.
\begin{prop}\cite{D}\label{prop2.3}
Let  $1<p<\infty, 1\leq \theta \leq\infty$ and $\
s\in(-\frac{N}{p}-1, \frac{N}{p}]$. Then there exists a generic
constant $C>0$ depending only on $s, N$ such that
$$\cases{\|[f,\dot{\Delta}_{q}]g\|_{L^{p}}\leq
Cc_{q}2^{-q(s+1)}\|f\|_{\dot{B}^{\frac{N}{p}+1}_{p,1}}\|g\|_{\dot{B}^{s}_{p,1}},\cr
\|[f,\dot{\Delta}_{q}]g\|_{L^{\theta}_{T}(L^{p})}\leq
Cc_{q}2^{-q(s+1)}\|f\|_{\widetilde{L}^{\theta_{1}}_{T}(\dot{B}^{\frac{N}{p}+1}_{p,1})}\|g\|_{\widetilde{L}^{\theta_{2}}_{T}(\dot{B}^{s}_{p,1})},}
$$
with $1/\theta=1/\theta_{1}+1/\theta_{2}$, where the commutator
$[\cdot,\cdot]$ is defined by $[f,g]=fg-gf$ and $\{c_{q}\}$ denotes
a sequence such that $\|(c_{q})\|_{ {l^{1}}}\leq 1$.
\end{prop}

\section{Entropy and Local well-posedness}\label{sec:3}
\setcounter{equation}{0}

First, let us introduce the energy function which is just an entropy
in the sense of Definition 2.1 of \cite{KY}:
$$\eta(\rho,\mathbf{m}):=\frac{|\mathbf{m}|^2}{2\rho}+h(\rho)\ \ \ \mbox{with}\ \
\ \mathbf{m}=\rho\textbf{v} \ \mbox{and}\ \
h'(\rho)=\int^{\rho}_{1}\frac{p'(s)}{s}ds.$$ For this rigorous
verification, see, \textit{e.g.}, \cite{XK}. Furthermore, the
associated entropy flux is
$$q(\rho,\mathbf{m})=\Big(\frac{|\mathbf{m}|^2}{2\rho^2}+\rho h'(\rho)\Big)\frac{\mathbf{m}}{\rho}.$$

Define
$$W=\left(
      \begin{array}{c}
        W_{1} \\
        W_{2} \\
      \end{array}
    \right):=\nabla\eta(\rho,\mathbf{m})=\left(
                             \begin{array}{c}
                               -\frac{|\mathbf{m}|^2}{2\rho^2}+h'(\rho) \\
                               \mathbf{m}/\rho \\
                             \end{array}
                           \right).
$$
Clearly, the mapping $U\rightarrow W$ is a diffeomorphism from
$\mathcal{O}_{(\rho,\mathbf{m})}:=\mathbb{R}^{+}\times
\mathbb{R}^{d}$ onto its range $\mathcal{O}_{W}$, and for classical
solutions $(\rho,\textbf{v})$ away from vacuum, (\ref{R-E1}) is
equivalent to the symmetric system
\begin{equation}
A^{0}(W)W_{t}+\sum_{j=1}^{d}A^{j}(W)W_{x_{j}}=H(W) \label{R-E9}
\end{equation}
with $$A^{0}(W)=\left(
          \begin{array}{cc}
            1 & W_{2}^{\top} \\
            W_{2} & W_{2}\otimes W_{2}+p'(\rho)I_{d} \\
          \end{array}
        \right),$$

$$A^{j}(W)=\left(
          \begin{array}{cc}
            W_{2j} & W_{2}^{\top}W_{2j}+p'(\rho)e_{j}^{\top} \\
            W_{2}W_{2j}+p'(\rho)e_{j} & W_{2j}(W_{2}\otimes W_{2}+p'(\rho)I_{d})+p'(\rho)(W_{2}\otimes e_{j}+e_{j}\otimes W_{2}) \\
          \end{array}
        \right),$$

$$H(W)=G(U(W))=p'(\rho)\left(
                 \begin{array}{c}
                   0 \\
                   -\frac{W_{2}}{\tau} \\
                 \end{array}
               \right)
,$$ where $I_{d}$ stands for the $d\times d$ unit matrix, and
$e_{j}$ is $d$-dimensional vector where the $j$th component is one,
others are zero. It follows from the definition of entropy variable
$W$ that $h'(\rho)=W_{1}+|W_{2}|^2/2$, so $p'(\rho)$ can be viewed
as a function of $W$, since $\rho$ is the function of
$W_{1}+|W_{2}|^2/2$, i.e. of $W$.

The corresponding initial data become into
\begin{equation}
W(0,x):=W_{0}=\Big(-\frac{|\textbf{v}_{0}|^2}{2}+h'(\rho_{0}),\textbf{v}_{0}\Big).\label{R-E10}
\end{equation}

In \cite{XK}, we have recently established a local well-posedness
theory for generally symmetric hyperbolic systems in the framework
of critical spaces, which is regarded as the generalization of the
classical existence theory of Kato and Majda \cite{K,M}. Actually,
the theory is also true in larger Besov spaces, whose the regularity
indices satisfy the condition (\ref{R-E666}). Hence, we can have the
following local existence result for the concrete problem
(\ref{R-E9})-(\ref{R-E10}).

\begin{prop} \label{prop3.1} For any fixed relaxation time $\tau>0$,
assume that the initial data $W_{0}$ satisfy $W_{0}-\bar{W}\in
B^{\sigma}_{2,r}(\bar{W}:=(h'(\bar{\rho}),\mathbf{0}))$ and take
values in a compact subset of $\mathcal{O}_{W}$, then, there exists
a time $T_{0}>0$ such that
\begin{itemize}
\item[(i)] Existence:  the
Cauchy problem (\ref{R-E9})-(\ref{R-E10}) has a unique classical
solution $W\in \mathcal{C}^{1}([0,T_{0}] \times \mathbb{R}^{d})$
satisfying
$$W-\bar{W}\in \widetilde{\mathcal{C}}_{T_{0}}(B^{\sigma}_{2,r})
\cap\widetilde{\mathcal{C}}^{1}_{T_{0}}(B^{\sigma-1}_{2,r});
$$

\item[(ii)] Blow-up criterion: there exists a constant $C_{0}>0$
such that the maximal time $T^{*}$ of existence of such a solution
can be bounded from below by
$T^{*}\geq\frac{C_{0}}{\|W_{0}-\bar{W}\|_{B^{\sigma}_{2,r}}}.$
Moreover, if $T^{*}$ is finite, then
$$\limsup_{t\rightarrow T^{*}}\|W-\bar{W}\|_{B^{\sigma}_{2,r}}=\infty$$ if and only if
$$\int^{T^{*}}_{0}\|\nabla W\|_{L^{\infty}}dt=\infty.$$
\end{itemize}
\end{prop}

\section{Global well-posedness}\label{sec:4}\setcounter{equation}{0}
\setcounter{equation}{0} This Section is devoted to the global
existence result in Theorem \ref{thm1.1}. To show that the solutions
of (\ref{R-E9})-(\ref{R-E10}), are globally defined, we need further
a priori estimates.

To do this, for any time $T>0$ and for any solution $W-\bar{W}\in
\widetilde{\mathcal{C}}_{T}(B^{\sigma}_{2,r})\cap\widetilde{\mathcal{C}}^{1}_{T}(B^{\sigma-1}_{2,r})$,
we define by $E(T)$ the energy functional and by $D_{\tau}(T)$ the
corresponding dissipation functional:
$$
E(T):=\|W-\bar{W}\|_{\widetilde{L}_{T}^\infty(B^{\sigma}_{2,r})},$$
$$D_{\tau}(T):=\frac{1}{\sqrt{\tau}}\|W_{2}\|_{\widetilde{L}_{T}^2(B^{\sigma}_{2,r})}
+\sqrt{\tau}\|\nabla W\|_{\widetilde{L}_{T}^2(B^{\sigma-1}_{2,r})},
$$
and $E(0):=\|W_{0}-\bar{W}\|_{B^{\sigma}_{2,r}}$. We also define $$
S(T):=\|W\|_{L^\infty([0,T]\times\mathbb{R}^{d})}+\|\nabla
W\|_{L^\infty([0,T]\times\mathbb{R}^{d})}.
$$
Note that the imbedding in Lemma \ref{lem2.2} and Remark
\ref{rem2.2}, we have $S(T)\leq CE(T)$ for some generic constant
$C>0$.

The next central task is to construct the desired a priori estimate,
which is included in the following proposition.
\begin{prop}\label{prop4.1}
Let $W$ be the solution of (\ref{R-E9})-(\ref{R-E10}) satisfying
$W-\bar{W}\in
\widetilde{\mathcal{C}}_{T}(B^{\sigma}_{2,r})\cap\widetilde{\mathcal{C}}^{1}_{T}(B^{\sigma-1}_{2,r})$.
If $W(t,x)$ takes values in a neighborhood of $\bar{W}$, then there
exists a non-decreasing continuous function
$C:\mathbb{R}^{+}\rightarrow \mathbb{R}^{+}$ which is independent of
$\tau$, such that the following nonlinear inequality holds:
\begin{eqnarray}
E(T)+D_{\tau}(T)\leq
C(S(T))\Big(E(0)+E(T)^{1/2}D_{\tau}(T)+E(T)D_{\tau}(T)\Big).\label{R-E11}
\end{eqnarray}
Furthermore, there exist some positive constants $\delta_{1},
\mu_{1}$ and $C_{1}$ independent of $\tau$, if $E(T)\leq
\delta_{1}$, then
\begin{eqnarray}E(T)+\mu_{1}D_{\tau}(T)\leq C_{1}E(0).\label{R-E12}
\end{eqnarray}

\end{prop}
The proof of Proposition \ref{prop4.1}, in fact, is to capture the
dissipation rates of $W=(W_{1},W_{2})$ in turn by using the Fourier
localization arguments. First, we give an elementary fact that
indicates the connection between the homogeneous and inhomogeneous
Chemin-Lerner spaces, which have been established in the recent work
\cite{XK}. Precisely, it reads as follows:
\begin{lem}\label{lem4.1}
Let $s>0, 1\leq\theta, p, r\leq+\infty$. When $\theta\geq r$, it
holds that
\begin{eqnarray}L^{\theta}_{T}(L^p)\cap\widetilde{L}^{\theta}_{T}(\dot{B}^{s}_{p,r})=\widetilde{L}^{\theta}_{T}(B^{s}_{p,r})\label{R-E43}\end{eqnarray}
for any $T>0$.
\end{lem}

Next we begin to prove the Proposition \ref{prop4.1}. For clarity,
we divide it into several lemmas.

\begin{lem} Under the assumptions stated in Proposition
\ref{prop4.1}, there exists a non-decreasing continuous function
$C:\mathbb{R}^{+}\rightarrow \mathbb{R}^{+}$ which is independent of
$\tau$, such that the following estimate holds:
\begin{eqnarray}
\|W-\bar{W}\|_{\widetilde{L}_{T}^\infty(B^{\sigma}_{2,r})}+\frac{1}{\sqrt{\tau}}\|W_{2}\|_{\widetilde{L}_{T}^2(B^{\sigma}_{2,r})}\leq
C(S(T))(E(0)+E(T)^{1/2}D_{\tau}(T)).\label{R-E13}
\end{eqnarray}
\end{lem}

\begin{proof} The proof is divided into three steps.\\

\textit{Step 1. The
$\widetilde{L}^\infty_{T}(\dot{B}^{\sigma}_{2,r})$ estimate of $W$
and the $\widetilde{L}^2_{T}(\dot{B}^{\sigma}_{2,r})$ one of
$W_{2}$}

Applying the homogeneous localization operator
$\dot{\Delta}_{q}(q\in \mathbb{Z})$ to (\ref{R-E9}), and taking the
inner product with $\dot{\Delta}_{q}$, we obtain
\begin{eqnarray}
&&A^{0}(W)\dot{\Delta}_{q}W_{t}+\sum_{j=1}^{d}A^{j}(W)\dot{\Delta}_{q}W_{x_{j}}+L\dot{\Delta}_{q}W\nonumber\\&=&
-[\dot{\Delta}_{q},A^{0}(W)]\partial_{t}W-\sum_{j=1}^{d}[\dot{\Delta}_{q},A^{j}(W)]\partial_{x_{j}}W+r(W)\label{R-E14}
\end{eqnarray}
with
$$L\dot{\Delta}_{q}W:=\frac{1}{\tau}\left(
                        \begin{array}{c}
                          0 \\
                          p'(\bar{\rho})\dot{\Delta}_{q}W_{2} \\
                        \end{array}
                      \right)\ \mbox{and}\ \ r(W):=\frac{1}{\tau}\left(
                        \begin{array}{c}
                          0 \\
                          \dot{\Delta}_{q}[(p'(\bar{\rho})-p'(\rho))W_{2}] \\
                        \end{array}
                      \right),
$$
where the commutator $[\cdot,\cdot]$ is defined by $[f,g]:=fg-gf$.

Take the $L^2$-inner product of (\ref{R-E14}) with
$\dot{\Delta}_{q}W$ to get
\begin{eqnarray}
\frac{1}{2}\langle
A^{0}(W)\dot{\Delta}_{q}W,\dot{\Delta}_{q}W\rangle_{t}+\frac{p'(\bar{\rho})}{\tau}\|\dot{\Delta}_{q}W_{2}\|^2_{L^2}
=\sum_{i=1}^{5}I_{i}(t),\label{R-E15}
\end{eqnarray}
where
$$I_{1}:=\frac{1}{2}\langle \partial_{t}A^{0}(W)\dot{\Delta}_{q}W,\dot{\Delta}_{q}W\rangle,\ \ \ I_{2}:=\frac{1}{2}\sum_{j=1}^{d}\langle\partial_{x_{j}}A^{j}(W)\dot{\Delta}_{q}W,\dot{\Delta}_{q}W\rangle,$$
$$I_{3}:=-\langle[\dot{\Delta}_{q},A^{0}(W)]\partial_{t}W,\dot{\Delta}_{q}W\rangle,\ \ I_{4}:=-\sum_{j=1}^{d}\langle
[\dot{\Delta}_{q},A^{j}(W)]\partial_{x_{j}}W,\dot{\Delta}_{q}W\rangle,\
\ I_{5}:=\langle r(W),\dot{\Delta}_{q}W\rangle.$$

In what follows, we begin to bound these nonlinear energy terms.
Firstly, remark that
\begin{eqnarray}
A^{0}(W)&=&\left(
           \begin{array}{cc}
             1 & W_{2}^{\top} \\
             W_{2} & W_{2}\otimes W_{2} \\
           \end{array}
         \right)+\left(
                   \begin{array}{cc}
                     0 & 0 \\
                    0 & p'(\rho)I_{d} \\
                   \end{array}
                 \right)\nonumber\\&:=&A^{0}_{I}(W_{2})+A^{0}_{II}(W),\label{R-E16}
\end{eqnarray}
we further get
\begin{eqnarray}
\partial_{t}A^{0}_{I}(W_{2})=\left(
                               \begin{array}{cc}
                                 0 & \partial_{t}W_{2}^{\top} \\
                                 \partial_{t}W_{2} & \partial_{t}W_{2}\otimes W_{2}+W_{2}\otimes \partial_{t}W_{2} \\
                               \end{array}
                             \right)\label{R-E17}
\end{eqnarray}
and
\begin{eqnarray}
\partial_{t}A^{0}_{II}(W)=\left(
                   \begin{array}{cc}
                     0 & 0 \\
                    0 & p''(\rho)(D_{W}\rho(W)\partial_{t}W)I_{d} \\
                   \end{array}
                 \right),\label{R-E18}
\end{eqnarray}
where $D_{W}$ stands for the (row) gradient operator with respect to
$W$. From the decomposition in (\ref{R-E16}) of $A^{0}(W)$, $I_{1}$
can be written as the sum
\begin{eqnarray}
I_{1}(t)=\frac{1}{2}\langle
\partial_{t}A^{0}_{I}(W_{2})\dot{\Delta}_{q}W,\dot{\Delta}_{q}W\rangle+\frac{1}{2}\langle
\partial_{t}A^{0}_{II}(W)\dot{\Delta}_{q}W,\dot{\Delta}_{q}W\rangle=:I_{11}(t)+I_{12}(t).\label{R-E19}
\end{eqnarray}
With the aid of the embedding properties in Lemma \ref{lem2.2} and
Remark \ref{rem2.1}, we obtain
\begin{eqnarray}
\Big|\int_{0}^{T}I_{11}(t)dt\Big|&=&\Big|\int_{0}^{T}\int_{\mathbb{R}^d}(\partial_{t}W_{2}\cdot\dot{\Delta}_{q}W_{2}\dot{\Delta}_{q}W_{1}
+\frac{1}{2}|\dot{\Delta}_{q}W_{2}|^2W_{2}\cdot\partial_{t}W_{2})\Big|\nonumber\\&\leq&C(S(T))\int_{0}^{T}\|\partial_{t}W_{2}\|_{L^\infty}
\|\dot{\Delta}_{q}W_{2}\|_{L^2}\|\dot{\Delta}_{q}W\|_{L^2}dt\nonumber\\&\leq&C(S(T))\int_{0}^{T}\Big(\|\nabla
W\|_{L^\infty}+\frac{1}{\tau}\|W_{2}\|_{L^\infty}\Big)\|\dot{\Delta}_{q}W_{2}\|_{L^2}\|\dot{\Delta}_{q}W\|_{L^2}dt
\nonumber\\&\leq&C(S(T))\|\nabla
W\|_{L_{T}^\infty(L^\infty)}\|\dot{\Delta}_{q}W_{2}\|_{L_{T}^2(L^2)}\|\dot{\Delta}_{q}W\|_{L_{T}^2(L^2)}
\nonumber\\&&+\frac{1}{\tau}\|\dot{\Delta}_{q}W\|_{L_{T}^\infty(L^2)}\|W_{2}\|_{L_{T}^2(L^\infty)}\|\dot{\Delta}_{q}W_{2}\|_{L_{T}^2(L^2)}.\label{R-E20}
\end{eqnarray}
Furthermore, multiplying the factor $2^{2q\sigma}$ on both sides of
(\ref{R-E20}) gives
\begin{eqnarray}
2^{2q\sigma}\Big|\int_{0}^{T}I_{11}(t)dt\Big|&\leq&
C(S(T))c_{q}^2\|W-\bar{W}\|_{\widetilde{L}_{T}^\infty(\dot{B}^{\sigma}_{2,r})}\Big(\|W_{2}\|_{\widetilde{L}_{T}^2(\dot{B}^{\sigma}_{2,r})}
\|\nabla
W\|_{\widetilde{L}_{T}^2(\dot{B}^{\sigma-1}_{2,r})}\nonumber\\&&+\frac{1}{\tau}\|W_{2}\|_{\widetilde{L}_{T}^2(\dot{B}^{\sigma-1}_{2,r})}
\|W_{2}\|_{\widetilde{L}_{T}^2(\dot{B}^{\sigma}_{2,r})}\Big).\label{R-E21}
\end{eqnarray}
Here and below, $\{c_{q}\}$ denotes some sequence which satisfies
$\|(c_{q})\|_{ {l^{r}}}\leq1$ although each $\{c_{q}\}$ is possibly
different in (\ref{R-E21}). Allow us to abuse the notation for
simplicity. Similarly, we have
\begin{eqnarray}
2^{2q\sigma}\Big|\int_{0}^{T}I_{12}(t)dt\Big|&\leq&C(S(T))2^{2q\sigma}
\int_{0}^{T}\int_{\mathbb{R}^d}|\partial_{t}W||\dot{\Delta}_{q}W_{2}|^2dxdt\nonumber\\&\leq&
C(S(T))2^{2q\sigma}\Big(\|\nabla
W\|_{L^\infty_{T}(L^\infty)}+\frac{1}{\tau}\|W_{2}\|_{L^\infty_{T}(L^\infty)}\Big)\|\dot{\Delta}_{q}W_{2}\|^2_{L^2_{T}(L^2)}
\nonumber\\&\leq&C(S(T))c_{q}^2\|W-\bar{W}\|_{\widetilde{L}_{T}^\infty(\dot{B}^{\sigma}_{2,r})}\frac{1}{\tau}\|W_{2}\|^2_{\widetilde{L}_{T}^2(\dot{B}^{\sigma}_{2,r})},\label{R-E22}
\end{eqnarray}
where we used the smallness of $\tau (0<\tau\leq1)$.

Combining (\ref{R-E21})-(\ref{R-E22}), it follows from the basic
fact in Lemma \ref{lem4.1} and Young's inequality that
\begin{eqnarray}
2^{2q\sigma}\Big|\int_{0}^{T}I_{1}(t)dt\Big|\leq
C(S(T))c_{q}^2E(T)D_{\tau}(T)^{2}.\label{R-E23}
\end{eqnarray}
Secondly, we turn to estimate $I_{2}$. For this purpose, we write
\begin{eqnarray}
A^{j}(W)&=&\left(
          \begin{array}{cc}
            W_{2j} & W_{2}^{\top}W_{2j} \\
            W_{2}W_{2j} & W_{2j}(W_{2}\otimes W_{2}) \\
          \end{array}
        \right)\nonumber\\&&+\left(
          \begin{array}{cc}
            0 & p'(\rho)e_{j}^{\top} \\
            p'(\rho)e_{j} & W_{2j}p'(\rho)I_{d}+p'(\rho)(W_{2}\otimes e_{j}+e_{j}\otimes W_{2}) \\
          \end{array}
        \right)\nonumber\\&:=&A^{j}_{I}(W_{2})+A^{j}_{II}(W).\label{R-E24}
\end{eqnarray}
Therefore, it is not difficult to get
\begin{eqnarray}
&&2^{2q\sigma}\Big|\int_{0}^{T}\frac{1}{2}\sum_{j=1}^{d}\langle\partial_{x_{j}}A^{j}_{I}(W_{2})\dot{\Delta}_{q}W,\dot{\Delta}_{q}W\rangle
dt\Big|\nonumber\\&\leq&C(S(T))2^{2q\sigma}\|\nabla
W_{2}\|_{L_{T}^2(L^\infty)}\|\dot{\Delta}_{q}W\|_{L_{T}^\infty(L^2)}\|\dot{\Delta}_{q}W\|_{L_{T}^2(L^2)}
\nonumber\\&\leq&C(S(T))c_{q}^2\|W-\bar{W}\|_{\widetilde{L}_{T}^\infty(\dot{B}^{\sigma}_{2,r})}\|W_{2}\|_{\widetilde{L}_{T}^2(\dot{B}^{\sigma}_{2,r})}
\|\nabla W\|_{\widetilde{L}_{T}^2(\dot{B}^{\sigma-1}_{2,r})}
\nonumber\\&\leq&C(S(T))c_{q}^2\|W-\bar{W}\|_{\widetilde{L}_{T}^\infty(B^{\sigma}_{2,r})}\Big(\frac{1}{\tau}\|W_{2}\|^2_{\widetilde{L}_{T}^2(B^{\sigma}_{2,r})}
+\tau\|\nabla
W\|^2_{\widetilde{L}_{T}^2(B^{\sigma-1}_{2,r})}\Big)\nonumber\\&\leq&C(S(T))c_{q}^2E(T)D_{\tau}(T)^{2}.\label{R-E25}
\end{eqnarray}
Thanks to the null block located at the first line and the first
column of $A^{j}_{II}(W)$, we are led to the estimate
\begin{eqnarray}
&&2^{2q\sigma}\Big|\int_{0}^{T}\frac{1}{2}\sum_{j=1}^{d}\langle\partial_{x_{j}}A^{j}_{II}(W_{2})\dot{\Delta}_{q}W,\dot{\Delta}_{q}W\rangle
dt\Big|\nonumber\\&\leq&C(S(T))2^{2q\sigma}\int_{0}^{T}\|\nabla
W\|_{L^\infty}\|\dot{\Delta}_{q}W_{2}\|_{L^2}\|\dot{\Delta}_{q}W\|_{L^2}dt
\nonumber\\&\leq&C(S(T))c_{q}^2\|W-\bar{W}\|_{\widetilde{L}_{T}^\infty(B^{\sigma}_{2,r})}\Big(\frac{1}{\tau}\|W_{2}\|^2_{\widetilde{L}_{T}^2(B^{\sigma}_{2,r})}
+\tau\|\nabla
W\|^2_{\widetilde{L}_{T}^2(B^{\sigma-1}_{2,r})}\Big)\nonumber\\&\leq&C(S(T))c_{q}^2E(T)D_{\tau}(T)^{2}.\label{R-E26}
\end{eqnarray}
Hence, together with (\ref{R-E25})-(\ref{R-E26}), we arrive at
\begin{eqnarray}
2^{2q\sigma}\Big|\int_{0}^{T}I_{2}dt\Big|\leq
C(S(T))c_{q}^2E(T)D_{\tau}(T)^{2}.\label{R-E27}
\end{eqnarray}

Thirdly, we also use the decomposition (\ref{R-E16}) of $A^{0}(W)$
in order to estimate the commutator occurring in $I_{3}$. Using the
commutator estimate (\ref{R-E59}) in Proposition \ref{prop5.1}, we
get
\begin{eqnarray}
&&2^{2q\sigma}\Big|\int_{0}^{T}-\langle[\dot{\Delta}_{q},A^{0}_{I}(W_{2})]\partial_{t}W,\dot{\Delta}_{q}W\rangle
dt\Big|\nonumber\\&\leq&2^{2q\sigma}\|[\dot{\Delta}_{q},A^{0}_{I}(W_{2})]\partial_{t}W\|_{L^1_{T}(L^2)}\|\dot{\Delta}_{q}W\|_{L^\infty_{T}(L^2)}
\nonumber\\&\leq& C2^{q\sigma}c_{q}\Big(\|\nabla
A^{0}_{I}(W_{2})\|_{L^2_{T}(L^\infty)}\|\partial_{t}W\|_{\widetilde{L}_{T}^2(\dot{B}^{\sigma-1}_{2,r})}\nonumber\\&&\hspace{5mm}+
\|\partial_{t}W\|_{L^2_{T}(L^\infty)}\|A^{0}_{I}(W_{2})\|_{\widetilde{L}_{T}^2(\dot{B}^{\sigma}_{2,r})}\Big)\|\dot{\Delta}_{q}W\|_{L^\infty_{T}(L^2)}
\nonumber\\&\leq&
C2^{q\sigma}c_{q}\|A^{0}_{I}(W_{2})\|_{\widetilde{L}_{T}^2(\dot{B}^{\sigma}_{2,r})}\|\partial_{t}W\|_{\widetilde{L}_{T}^2(\dot{B}^{\sigma-1}_{2,r})}
\|\dot{\Delta}_{q}W\|_{L^\infty_{T}(L^2)}
\nonumber\\&\leq&C(S(T))c^2_{q}\|W-\bar{W}\|_{\widetilde{L}_{T}^\infty(\dot{B}^{\sigma}_{2,r})}\nonumber\\&&\hspace{5mm}\times\Big(\|\nabla
W\|_{\widetilde{L}_{T}^2(\dot{B}^{\sigma-1}_{2,r})}+\frac{1}{\tau}\|W_{2}\|_{\widetilde{L}_{T}^2(\dot{B}^{\sigma-1}_{2,r})}\Big)\|W_{2}\|_{\widetilde{L}_{T}^2(\dot{B}^{\sigma}_{2,r})}
\nonumber\\&\leq&C(S(T))c_{q}^2E(T)D_{\tau}(T)^{2},\label{R-E28}
\end{eqnarray}
where $(\sigma, r)$ satisfies the condition (\ref{R-E666}) and we
take  $\theta_{1}=\theta_{2}=\theta_{3}=\theta_{4}=2$ in
(\ref{R-E59}).

Recall that $p'(\rho)$ can be viewed as a function of $W$ and the
nice construction of $A^{0}_{II}(W)$, we gather that
\begin{eqnarray}
&&2^{2q\sigma}\Big|\int_{0}^{T}-\langle[\dot{\Delta}_{q},A^{0}_{II}(W)]\partial_{t}W,\dot{\Delta}_{q}W\rangle
dt\Big|\nonumber\\&\leq&2^{2q\sigma}\int_{0}^{T}\|[\dot{\Delta}_{q},p'(\rho)I_{d}]\partial_{t}W_{2}\|_{L^2}\|\dot{\Delta}_{q}W_{2}\|_{L^2}dt
\nonumber\\&\leq&C2^{q\sigma}c_{q}\Big(\|\nabla
p'(\rho)\|_{L^\infty_{T}(L^\infty)}\|\partial_{t}W_{2}\|_{\widetilde{L}_{T}^2(\dot{B}^{\sigma-1}_{2,r})}+
\|\partial_{t}W_{2}\|_{L^2(L^\infty)}\|p'(\rho)\|_{\widetilde{L}_{T}^\infty(\dot{B}^{\sigma}_{2,r})}\Big)
\nonumber\\&\leq&
C2^{q\sigma}c_{q}\|p'(\rho)\|_{\widetilde{L}_{T}^\infty(\dot{B}^{\sigma}_{2,r})}\|\partial_{t}W_{2}\|_{\widetilde{L}_{T}^2(\dot{B}^{\sigma-1}_{2,r})}\|\dot{\Delta}_{q}W_{2}\|_{L^2_{T}(L^2)}
\nonumber\\&\leq&C(S(T))c_{q}^2E(T)D_{\tau}(T)^{2},\label{R-E29}
\end{eqnarray}
where we used the commutator estimates in Proposition \ref{prop5.1}
and the homogeneous case of Proposition \ref{prop2.2}.

Therefore, we finally get
\begin{eqnarray}
2^{2q\sigma}\Big|\int_{0}^{T}I_{3}dt\Big|\leq
C(S(T))c_{q}^2E(T)D_{\tau}(T)^{2}.\label{R-E30}
\end{eqnarray}

The estimates of commutators arising in $I_{4}$ are actually simpler
than (\ref{R-E28})-(\ref{R-E29}), since they only involve spatial
derivatives. Let us give the simplified steps only:
\begin{eqnarray}
&&2^{2q\sigma}\Big|\int_{0}^{T}-\sum_{j=1}^{d}\langle
[\dot{\Delta}_{q},A^{j}_{I}(W_{2})]\partial_{x_{j}}W,\dot{\Delta}_{q}W\rangle\Big|dt\nonumber\\&\leq&C(S(T))c_{q}^2
\|W-\bar{W}\|_{\widetilde{L}_{T}^\infty(\dot{B}^{\sigma}_{2,r})}\|W_{2}\|_{\widetilde{L}_{T}^2(\dot{B}^{\sigma}_{2,r})}
\|\nabla
W\|_{\widetilde{L}_{T}^2(\dot{B}^{\sigma-1}_{2,r})}\nonumber\\&\leq&
C(S(T))c_{q}^2E(T)D_{\tau}(T)^{2},\label{R-E31}
\end{eqnarray}
and
\begin{eqnarray}
&&2^{2q\sigma}\Big|\int_{0}^{T}-\sum_{j=1}^{d}\langle
[\dot{\Delta}_{q},A^{j}_{II}(W)]\partial_{x_{j}}W,\dot{\Delta}_{q}W\rangle\Big|dt\nonumber\\&\leq&C2^{2q\sigma}
\int_{0}^{T}\Big(\|[\dot{\Delta}_{q},A^{j}_{II1}(W)]\partial_{x_{j}}W_{2}\|\|\dot{\Delta}_{q}W\|_{L^2}+
\|[\dot{\Delta}_{q},A^{j}_{II2}(W)]\partial_{x_{j}}W\|\|\dot{\Delta}_{q}W_{2}\|_{L^2}\Big)dt
\nonumber\\&\leq&C(S(T))c_{q}^2
\|W-\bar{W}\|_{\widetilde{L}_{T}^\infty(\dot{B}^{\sigma}_{2,r})}\|W_{2}\|_{\widetilde{L}_{T}^2(\dot{B}^{\sigma}_{2,r})}
\|\nabla
W\|_{\widetilde{L}_{T}^2(\dot{B}^{\sigma-1}_{2,r})}\nonumber\\&\leq&
C(S(T))c_{q}^2E(T)D_{\tau}(T)^{2},\label{R-E32}
\end{eqnarray}
where $A^{j}_{II1}(W)$ and $A^{j}_{II2}(W)$ can be viewed as the
smooth functions of $W$.

Hence, from (\ref{R-E31})-(\ref{R-E32}), we deduce that
\begin{eqnarray}
2^{2q\sigma}\Big|\int_{0}^{T}I_{4}dt\Big|\leq
C(S(T))c_{q}^2E(T)D_{\tau}(T)^{2}.\label{R-E33}
\end{eqnarray}
Finally, due to the homogeneous cases of Propositions
\ref{prop2.1}-\ref{prop2.2}, $I_{5}$ can be estimated as
\begin{eqnarray}
2^{2q\sigma}\Big|\int_{0}^{T}I_{5}dt\Big|&\leq&\frac{2^{2q\sigma}}{\tau}\int_{0}^{T}\|\dot{\Delta}_{q}[(p'(\bar{\rho})-p'(\rho))W_{2}]
\|_{L^2}\|\dot{\Delta}_{q}W_{2}\|_{L^2}dt\nonumber\\&\leq&\frac{c_{q}^2}{\tau}\|(p'(\bar{\rho})-p'(\rho))W_{2}\|_{\widetilde{L}_{T}^2(\dot{B}^{\sigma}_{2,r})}
\|W_{2}\|_{\widetilde{L}_{T}^2(\dot{B}^{\sigma}_{2,r})}\nonumber\\&\leq&C(S(T))\frac{c_{q}^2}{\tau}
\|W-\bar{W}\|_{\widetilde{L}_{T}^\infty(\dot{B}^{\sigma}_{2,r})}\|W_{2}\|^2_{\widetilde{L}_{T}^2(\dot{B}^{\sigma}_{2,r})}\nonumber\\&\leq&C(S(T))c_{q}^2E(T)D_{\tau}(T)^{2}.\label{R-E34}
\end{eqnarray}

In addition, we define
$$\|f\|_{L^2_{\tilde{A}^{0}}}:=\langle A^{0}(W)f,f\rangle^{1/2}.$$
Clearly, the norm $\|f\|_{L^2_{A^{0}}}\thickapprox\|f\|_{L^2}$,
since $W(t,x)$ takes values in a neighborhood of $\bar{W}$.
Consequently,  according to inequalities (\ref{R-E23}),
(\ref{R-E27}), (\ref{R-E30}) and (\ref{R-E33})-(\ref{R-E34}), we
conclude that
\begin{eqnarray}
&&2^{2q\sigma}\|\dot{\Delta}_{q}(W(t)-\bar{W})\|^2_{L^2}+\frac{2^{2q\sigma}}{\tau}\|\dot{\Delta}_{q}W_{2}\|^2_{L^2_{t}(L^2)}\nonumber\\&\leq&
C(S(T))2^{2q\sigma}\|\dot{\Delta}_{q}(W_{0}-\bar{W})\|^2_{L^2}+C(S(T))c_{q}^2E(T)D_{\tau}(T)^{2}.\label{R-E35}
\end{eqnarray}
Then it follows from the classical Young's inequality that
\begin{eqnarray}
&&2^{q\sigma}\|\dot{\Delta}_{q}(W-\bar{W})\|_{L^\infty_{T}(L^2)}+\frac{2^{q\sigma}}{\sqrt{\tau}}\|\dot{\Delta}_{q}W_{2}\|_{L^2_{T}(L^2)}\nonumber\\&\leq&
C(S(T))2^{q\sigma}\|\dot{\Delta}_{q}(W_{0}-\bar{W})\|_{L^2}+C(S(T))c_{q}E(T)^{1/2}D_{\tau}(T).\label{R-E36}
\end{eqnarray}
Hence, taking the $\ell^{r}$-norm on $q\in \mathbb{Z}$ on both sides
of (\ref{R-E36}) gives immediately
\begin{eqnarray}
\|W-\bar{W}\|_{\widetilde{L}_{T}^\infty(\dot{B}^{\sigma}_{2,r})}+\frac{1}{\sqrt{\tau}}\|W_{2}\|_{\widetilde{L}_{T}^2(\dot{B}^{\sigma}_{2,r})}\leq
C(S(T))(E(0)+E(T)^{1/2}D_{\tau}(T)).\label{R-E37}
\end{eqnarray}\\

\textit{Step 2. The $L^\infty_{T}(L^2)$ estimate of $W$ and the
$L^2_{T}(L^2)$ one of $W_{2}$}

It is convenient to introduce the relative entropy
$\tilde{\eta}(\rho,\mathbf{m})$ by
\begin{eqnarray}
\tilde{\eta}(\rho,\mathbf{m}):=\eta(\rho,\mathbf{m})-\eta(\bar{\rho},\mathbf{0})-D_{\rho}
\eta(\bar{\rho},\mathbf{0})(\rho-\bar{\rho}).\label{R-E38}
\end{eqnarray}
Then the strictly convex quantity $\tilde{\eta}(\rho,\mathbf{m})$
satisfies
\begin{eqnarray}\tilde{\eta}(\rho,\mathbf{m})\geq0,\ \ \tilde{\eta}(\bar{\rho},\mathbf{0})=0,\ \  D_{\rho}\tilde{\eta}(\bar{\rho},\mathbf{0})=0.
\label{R-E39}
\end{eqnarray}
Furthermore, $\tilde{\eta}(U)$ is equivalent to the quadratic
function $|\rho-\bar{\rho}|^2+|\mathbf{m}|^2$ and hence to
$|W-\bar{W}|^2$, since $|\rho-\bar{\rho}|+|\mathbf{m}|$ lies in a
bounded set. Accordingly, the entropy flux
$q_{j}(\rho,\mathbf{m})(j=1,2.\cdot\cdot\cdot,d)$ are modified as
follows
\begin{eqnarray}
\tilde{q}_{j}(\rho,\mathbf{m})=q_{j}(\rho,\mathbf{m})-D_{\rho}
\eta(\bar{\rho},\mathbf{0})\mathbf{m}_{j},\label{R-E40}
\end{eqnarray}
and we get the entropy-entropy flux equation
\begin{eqnarray}
\partial_{t}\tilde{\eta}(\rho,\mathbf{m})+\sum_{j=1}^{d}\partial_{x_{j}}\tilde{q}^{j}(\rho,\mathbf{m})=-\frac{1}{\tau}\rho|\mathbf{v}|^2.\label{R-E41}
\label{R-E44}
\end{eqnarray}
Integrating (\ref{R-E41}) over $[0,T]\times \mathbb{R}^{N}$ implies
immediately
\begin{eqnarray}
\|W-\bar{W}\|_{L^\infty_{T}(L^2)}+\frac{1}{\sqrt{\tau}}\|W_{2}\|_{L^2_{T}(L^2)}&\leq&
C(S(T))\|W_{0}-\bar{W}\|_{L^2}\nonumber\\&\leq&
C(S(T))(E(0)+E(T)^{1/2}D_{\tau}(T)).\label{R-E42}
\end{eqnarray}

\textit{Step 3. Combining the above analysis.}

Note that the fact (\ref{R-E43}), (\ref{R-E13}) is followed by
(\ref{R-E37}) and (\ref{R-E42}).
\end{proof}

The next step consists in deriving the
$\widetilde{L}^2_{T}(B^{\sigma-1}_{2,r})$ estimate of $\nabla W$. To
this end, we use an important stability condition
(``Shizuta-Kawashima" condition), which has been first formulated in
\cite{SK} by the second author, and well developed in
\cite{KY,XK,Y}.  As shown by the recent work \cite{XK}, the
compressible Euler equations with relaxation satisfies the stability
condition at the constant state naturally. Furthermore, we obtain
the concrete content of the compensating matrix $K(\xi)$, which has
been pointed out by \cite{CG,LC}. Precisely,
\begin{lem}[Shizuta-Kawashima] \label{lem4.2}For all ~$\xi\in \mathbb{R}^{N},\
\xi\neq0$, there exists a $(N+1)\times(N+1)$ matrix $K(\xi)$
depending smooth on the unit sphere $\mathbb{S}^{N-1}$:
\begin{eqnarray*}
K(\xi)=\left(%
\begin{array}{cc}
  0 & \frac{1}{p'(\bar{\rho})}\frac{\xi^\top}{|\xi|} \\
  -\frac{\xi}{|\xi|} & 0 \\
\end{array}%
\right),
\end{eqnarray*}
such that $K(\xi)A^{0}(\bar{W})$ is skew-symmetric and
\begin{eqnarray}
K(\xi)\sum_{j=1}^{N}\xi_{j}A^{j}(\bar{W})=\left(%
\begin{array}{cc}
  |\xi| & 0 \\
  0 & -p'(\bar{\rho})\frac{\xi\otimes\xi}{|\xi|} \\
\end{array}%
\right),\label{R-E44}
\end{eqnarray}
where $A^0$ and $A^{j}$ are the matrices appearing in the symmetric
system (\ref{R-E9}).
\end{lem}

Next we shall use the skew-symmetry of the compensating matrix
$K(\xi)$ in the Fourier spaces to establish the estimate of $\nabla
W$.
\begin{lem}\label{lem4.3}Under the assumptions stated in Proposition
\ref{prop4.1}, there exists a non-decreasing continuous function
$C:\mathbb{R}^{+}\rightarrow \mathbb{R}^{+}$ which is independent of
$\tau$, such that the following estimate holds:
\begin{eqnarray}
\sqrt{\tau}\|\nabla W\|_{\widetilde{L}^2(B^{\sigma-1}_{2,r})}\leq
C(S(T))\Big(E(0)+E(T)^{1/2}D_{\tau}(T)+E(T)D_{\tau}(T)\Big).\label{R-E45}
\end{eqnarray}
\end{lem}
\begin{proof}
Firstly, we linearize (\ref{R-E9}) around the constant state
$\bar{W}$:
\begin{eqnarray}
A^{0}(\bar{W})\partial_{t}\widetilde{W}+\sum_{j=1}^{d}A^{j}(\bar{W})\partial_{x_{j}}\widetilde{W}=-LW+\mathcal{G}\label{R-E46}
\end{eqnarray}
with $\widetilde{W}:=W-\bar{W}$, where $$LW:=\frac{1}{\tau}\left(
              \begin{array}{c}
                0 \\
                p'(\bar{\rho})W_{2} \\
              \end{array}
            \right)
$$
and
\begin{eqnarray}\mathcal{G}&:=&-\sum_{j=1}^{d}A^{0}(\bar{W})[A^{0}(W)^{-1}A^{j}(W)-A^{0}(\bar{W})^{-1}A^{j}(\bar{W})]\partial_{x_{j}}W
\nonumber\\&&-\frac{1}{\tau}A^{0}(\bar{W})[p'(\rho)A^{0}(W)^{-1}-p'(\bar{\rho})A^{0}(\bar{W})^{-1}]\left(
              \begin{array}{c}
                0 \\
                W_{2} \\
              \end{array}
            \right).\label{R-E47}
\end{eqnarray}
Applying the \textit{inhomogeneous} localization operator
$\Delta_{q}(q\geq-1)$ to  (\ref{R-E46}) implies
\begin{eqnarray}
A^{0}(\bar{W})\partial_{t}\Delta_{q}\widetilde{W}+\sum^{d}_{j=1}A^{j}(\bar{W})\partial_{x_{j}}\Delta_{q}\widetilde{W}=-L\Delta_{q}W+\Delta_{q}\mathcal{G}.\label{R-E48}
\end{eqnarray}
Take the Fourier transform (in the space variable $x$), multiply by
$-i\tau(\widehat{\Delta_{q}\widetilde{W}})^{\ast}K(\xi)$($^{\ast}$
transposed conjugate), and compute the real part of each term in the
resulting equality to get
\begin{eqnarray}
&&\tau \mathrm{Im}
\Big((\widehat{\Delta_{q}\widetilde{W}})^{\ast}(K(\xi)A^{0}(\bar{W}))\frac{d}{dt}\widehat{\Delta_{q}\widetilde{W}}\Big)+\tau(\widehat{\Delta_{q}\widetilde{W}})^{\ast}K(\xi)
\Big(\sum_{j=1}^{d}\xi_{j}A_{j}(\bar{W})\Big)\widehat{\Delta_{q}\widetilde{W}}\nonumber\\
&=&
-p'(\bar{\rho})\mathrm{Im}\Big((\widehat{\Delta_{q}\widetilde{W}_{1}})^{\ast}\frac{\xi^{\top}}{|\xi|}\widehat{\Delta_{q}W_{2}}\Big)
+\tau\mathrm{Im}\Big((\widehat{\Delta_{q}\widetilde{W}})^{\ast}K(\xi)\widehat{\Delta_{q}\mathcal{G}}\Big).\label{R-E49}
\end{eqnarray}
Using the skew-symmetry of $K(\xi)A^{0}(\bar{W})$, we have
\begin{eqnarray}
\mathrm{Im}
\Big((\widehat{\Delta_{q}\widetilde{W}})^{\ast}(K(\xi)A^{0}(\bar{W}))\frac{d}{dt}\widehat{\Delta_{q}\widetilde{W}}\Big)=\frac{1}{2}\frac{d}{dt}\mathrm{Im}
\Big((\widehat{\Delta_{q}\widetilde{W}})^{\ast}(K(\xi)A^{0}(\bar{W}))\widehat{\Delta_{q}\widetilde{W}}\Big).\label{R-E50}
\end{eqnarray}
Then, with the help of (\ref{R-E44}), the left-hand of (\ref{R-E49})
is bounded from below by
\begin{eqnarray}
\frac{\tau}{2}\frac{d}{dt}\mathrm{Im}
\Big((\widehat{\Delta_{q}\widetilde{W}})^{\ast}(K(\xi)A^{0}(\bar{W}))\widehat{\Delta_{q}\widetilde{W}}\Big)+\tau
p'(\bar{\rho})|\xi||\widehat{\Delta_{q}\widetilde{W}}|^2-C\tau|\xi||\widehat{\Delta_{q}W_{2}}|^2,\label{R-E51}
\end{eqnarray}
where $C>0$ is a generic constant independent of $\tau$. By Young's
inequality, the right-hand of (\ref{R-E49}) is dominated by
\begin{eqnarray}
\frac{1}{2}\tau
p'(\bar{\rho})|\xi||\widehat{\Delta_{q}\widetilde{W}}|^2+\frac{C}{\tau|\xi|}|\widehat{\Delta_{q}W_{2}}|^2+\frac{C\tau}{|\xi|}|\widehat{\Delta_{q}\mathcal{G}}|^2.\label{R-E52}
\end{eqnarray}
Hence, combine the inequalities (\ref{R-E51})-(\ref{R-E52}) to
deduce that
\begin{eqnarray}
\frac{1}{2}\tau
p'(\bar{\rho})|\xi||\widehat{\Delta_{q}\widetilde{W}}|^2&\lesssim&
\frac{1}{\tau}\Big(|\xi|+\frac{1}{|\xi|}\Big)|\widehat{\Delta_{q}W_{2}}|^2+\frac{\tau}{|\xi|}|\widehat{\Delta_{q}\mathcal{G}}|^2
\nonumber\\&&-\frac{\tau}{2}\frac{d}{dt}\mathrm{Im}
\Big((\widehat{\Delta_{q}\widetilde{W}})^{\ast}(K(\xi)A^{0}(\bar{W}))\widehat{\Delta_{q}\widetilde{W}}\Big).\label{R-E53}
\end{eqnarray}
Multiplying (\ref{R-E53}) by $|\xi|$ and integrating it over
$[0,T]\times \mathbb{R}^{N}$, then by Plancherel's theorem, we
arrive at
\begin{eqnarray}
\tau\|\Delta_{q}\nabla W\|^2_{L^2_{T}(L^2)}&\lesssim&
(\|\Delta_{q}\widetilde{W}_{0}\|^2_{L^2}+\|\Delta_{q}\nabla
\widetilde{W}_{0}\|^2_{L^2})+(\|\Delta_{q}\widetilde{W}\|^2_{L^\infty_{T}(L^2)}+\|\Delta_{q}\nabla
\widetilde{W}\|^2_{L^\infty_{T}(L^2)})\nonumber\\&&+\frac{1}{\tau}\Big(\|\Delta_{q}W_{2}\|^2_{L^2_{T}(L^2)}+\|\Delta_{q}\nabla
W_{2}\|^2_{L^2_{T}(L^2)}\Big)+\tau\|\Delta_{q}\mathcal{G}\|^2_{L^2_{T}(L^2)},\label{R-E54}
\end{eqnarray}
where we used the uniform boundedness of the matrix
$K(\xi)A^{0}(\bar{W})(\xi\neq0)$ and the smallness of $\tau
(0<\tau\leq1)$.

By multiplying the factor $2^{2q(\sigma-1)}$ and taking the
$\ell^{r}$-norm on $q\geq-1$ on both sides of (\ref{R-E54}), then
extracting the square root of the resulting inequality, we are led
to
\begin{eqnarray}
\sqrt{\tau}\|\nabla
W\|_{\widetilde{L}^2(B^{\sigma-1}_{2,r})}&\lesssim&\|W_{0}-\bar{W}\|_{B^{\sigma}_{2,r}}+\|W-\bar{W}\|_{\widetilde{L}^\infty(B^{\sigma}_{2,r})}
\nonumber\\&&+\frac{1}{\sqrt{\tau}}\|W_{2}\|_{\widetilde{L}^2(B^{\sigma}_{2,r})}+\sqrt{\tau}\|\mathcal{G}\|_{\widetilde{L}^2(B^{\sigma-1}_{2,r})}.\label{R-E55}
\end{eqnarray}
Recalling the definition in (\ref{R-E44}) of $\mathcal{G}$, with the
aid of Propositions \ref{prop2.1}-\ref{prop2.2}, we obtain
\begin{eqnarray}
\|\mathcal{G}\|_{\widetilde{L}^2(B^{\sigma-1}_{2,r})}\leq
C(S(T))\|W-\bar{W}\|_{\widetilde{L}^\infty(B^{\sigma}_{2,r})}\Big(\frac{1}{\tau}\|W_{2}\|_{\widetilde{L}^2(B^{\sigma}_{2,r})}+\|\nabla
W\|_{\widetilde{L}^2(B^{\sigma-1}_{2,r})}\Big).\label{R-E56}
\end{eqnarray}
Therefore, we get ultimately
\begin{eqnarray}
\sqrt{\tau}\|\nabla
W\|_{\widetilde{L}^2(B^{\sigma-1}_{2,r})}&\leq&C(S(T))\Big\{\|W_{0}-\bar{W}\|_{B^{\sigma}_{2,r}}+\|W-\bar{W}\|_{\widetilde{L}^\infty(B^{\sigma}_{2,r})}
+\frac{1}{\sqrt{\tau}}\|W_{2}\|_{\widetilde{L}^2(B^{\sigma}_{2,r})}\nonumber\\&&+\|W-\bar{W}\|_{\widetilde{L}^\infty(B^{\sigma}_{2,r})}\Big(\frac{1}{\sqrt{\tau}}\|W_{2}\|_{\widetilde{L}^2(B^{\sigma}_{2,r})}
+\sqrt{\tau}\|\nabla
W\|_{\widetilde{L}^2(B^{\sigma-1}_{2,r})}\Big)\Big\}\nonumber\\&\leq&C(S(T))\Big(E(0)+E(T)^{1/2}D_{\tau}(T)+E(T)D_{\tau}(T)\Big),\label{R-E57}
\end{eqnarray}
which completes the proof of Lemma \ref{lem4.3}.
\end{proof}

\noindent\textit{\underline{Proof of Proposition \ref{prop4.1}.}} By
combining (\ref{R-E13}) and (\ref{R-E45}), we conclude the nonlinear
inequality (\ref{R-E11}). In addition, the inequality (\ref{R-E12})
follows from (\ref{R-E11}) and the a priori assumption
$E(T)\leq\delta_{1}$ readily. Hence the proof of Proposition
\ref{prop4.1} is finished. $\square$

Thanks to the standard boot-strap argument, for instance, see
[\cite{MN}, Theorem 7.1] or the outline given in \cite{XW}, we
extend the local-in-time solutions in Proposition \ref{prop3.1} to
the global-in-time classical solutions of
(\ref{R-E9})-(\ref{R-E10}). Furthermore, we arrive at Theorem
\ref{thm1.1}.

The proof of Theorem \ref{thm1.2} is just the same as that in
\cite{XW}, we thus feel free to skip the details, the interested
readers is referred to \cite{XW}.

\section{Appendix}\label{sec:5}\setcounter{equation}{0}
In this section, we present a revision of commutator estimates in
Proposition \ref{prop2.3} and relax the restriction on the couple
$(s,r)$ in comparison with the commutator estimate in Proposition
\ref{prop2.3}, which enables us to construct the a priori estimates
in more general Besov spaces.

\begin{prop}\label{prop5.1}
For $s>-1$, $1\leq p\leq\infty$ and $ 1\leq r\leq\infty$, there is a
constant $C>0$ such that
\begin{eqnarray}
\|[f,\dot{\Delta}_{q}]g\|_{L^{p}} \leq
Cc_{q}2^{-q(s+1)}\Big(\|\nabla
f\|_{L^\infty}\|g\|_{\dot{B}^{s}_{p,r}}+\|g\|_{L^{p_1}}\|f\|_{\dot{B}^{s+1}_{p_2,r}}\Big)\label{R-E58}
\end{eqnarray}
and
\begin{eqnarray}
\|[f,\dot{\Delta}_{q}]g\|_{L^{\theta}_{T}(L^{p})} \leq
Cc_{q}2^{-q(s+1)}\Big(\|\nabla
f\|_{L^{\theta_1}_{T}(L^\infty)}\|g\|_{\widetilde{L}^{\theta_2}_{T}(\dot{B}^{s}_{p,r})}+\|g\|_{L^{\theta_3}_{T}(L^{p_1})}\|f\|_{\widetilde{L}^{\theta_4}_{T}(\dot{B}^{s+1}_{p_2,r})}\Big),\label{R-E59}
\end{eqnarray}
where $1/p=1/p_{1}+1/p_{2}$,
$1/\theta=1/\theta_{1}+1/\theta_{2}=1/\theta_{3}+1/\theta_{4}$ and
$c_{q}$ denotes a sequence such that $\|(c_{q})\|_{\ell^{r}}\leq1$.
\end{prop}
\begin{proof}
We decompose
\begin{eqnarray}
[f,\dot{\Delta}_{q}]g:=K_{1}+K_{2}+K_{3}+K_{4}+K_{5}
\label{R-E60}\end{eqnarray} with
$$K_{1}=[T_{f},\dot{\Delta}_{q}]g,\ \ K_{2}=R(f,\dot{\Delta}_{q}g) \ \ K_{3}=-\dot{\Delta}_{q}R(f,g),$$
$$K_{4}=T_{\dot{\Delta}_{q}g}f,\ \ K_{5}=-\dot{\Delta}_{q}T_{g}f,$$
where $T$ and $R$ stand for the paraproduct and remainder operators
(see \cite{B} by J.-M. Bony),  which are given by
$T_{f}g=\sum_{q'\leq
q-2}\dot{\Delta}_{q'}f\dot{\Delta}_{q}g=\sum_{q\in\mathbb{Z}}S_{q-1}f\dot{\Delta}_{q}g$
and $R(f,g)=\sum_{|q-q'|\leq1}\dot{\Delta}_{q}f\dot{\Delta}_{q'}g.$
(\ref{R-E60}) is obtained after noticing that
$$fg=T_{f}g+T_{g}f+R(f,g). $$
From the definition of $\dot{\Delta}_{q}$, we have the almost
orthogonal properties:
$$\dot{\Delta}_{p}\dot{\Delta}_{q}f\equiv 0 \ \ \ \mbox{if}\ \ \ |p-q|\geq 2,$$
$$\dot{\Delta}_{q}(S_{q-1}f\dot{\Delta}_{p}g)\equiv 0\ \ \ \mbox{if}\ \ \ |p-q|\geq 5.$$

Now, we go back to the proof of inequality (\ref{R-E58}). For
$K_{1}$, it follows from the Taloy's formula of first order and
Young's inequality that
\begin{eqnarray}
\|K_{1}\|_{L^p}&\leq& C\sum_{|q-k|\leq4}2^{-q}\|S_{k-1}\nabla
f\|_{L^\infty}\|\dot{\Delta}_{k}g\|_{L^p}\nonumber\\&\leq&
Cc_{q}2^{-q(s+1)}\|\nabla f\|_{L^\infty}\|g\|_{\dot{B}^{s}_{p,r}}.
\label{R-E61}
\end{eqnarray}

For $K_{2}$, we have
\begin{eqnarray*}
K_{2}&=&\sum_{k\in
\mathbb{Z}}\sum_{|k-k'|\leq1}(\dot{\Delta}_{k}f)(\dot{\Delta}_{k'}\dot{\Delta}_{q}g)
\nonumber\\&=&\sum_{|k-q|\leq2}\sum_{|k-k'|\leq1}(\dot{\Delta}_{k}f)(\dot{\Delta}_{k'}\dot{\Delta}_{q}g).
\end{eqnarray*}
Then by H\"{o}lder inequalities, we obtain
\begin{eqnarray}
\|K_{2}\|_{L^p}&\leq&
C\sum_{|k-q|\leq2}\sum_{|k-k'|\leq1}\|\dot{\Delta}_{k}f\|_{L^{p_{2}}}\|\dot{\Delta}_{k'}\dot{\Delta}_{q}g\|_{L^{p_{1}}}
\nonumber\\&\leq&C\|g\|_{L^{p_{1}}}\sum_{|k-q|\leq2}\|\dot{\Delta}_{k}f\|_{L^{p_{2}}}
\nonumber\\&\leq&
Cc_{q}2^{-q(s+1)}\|f\|_{\dot{B}^{s+1}_{p_2,r}}\|g\|_{L^{p_{1}}},\label{R-E62}
\end{eqnarray}
where $1/p=1/p_{1}+1/p_{2}$.

For $K_{3}$, we have
\begin{eqnarray*}
K_{3}&=&-\dot{\Delta}_{q}\Big(\sum_{k\in
\mathbb{Z}}\sum_{|k-k'|\leq1}\dot{\Delta}_{k}f\dot{\Delta}_{k'}g\Big)
\nonumber\\&=&-\sum_{\max(k,k')\geq
q-2}\sum_{|k-k'|\leq1}\dot{\Delta}_{q}(\dot{\Delta}_{k}f\dot{\Delta}_{k'}g)
\end{eqnarray*}
Using the H\"{o}lder inequality and Young's inequality for
sequences, we proceed $K_{3}$ as follows:
\begin{eqnarray}
\|K_{3}\|_{L^p}&\leq& C\sum_{\max(k,k')\geq q-2}
\sum_{|k-k'|\leq1}\|\dot{\Delta}_{k}f\|_{L^{p_{2}}}\|\dot{\Delta}_{k'}g\|_{L^{p_{1}}}
\nonumber\\&\leq& C\|g\|_{L^{p_{1}}}\sum_{k\geq
q-3}\|\dot{\Delta}_{k}f\|_{L^{p_{2}}}\nonumber\\
&\leq& C
c_{q}2^{-q(s+1)}\|f\|_{\dot{B}^{s+1}_{p_2,r}}\|g\|_{L^{p_{1}}},\label{R-E63}
\end{eqnarray}
where $s+1>0$ is required.

For $K_{4}$, we have
\begin{eqnarray*}
K_{4}=-\sum_{k\geq q+1}S_{k-1}\dot{\Delta}_{q}g\dot{\Delta}_{k}f.
\end{eqnarray*}
In a similar manner as $K_{3}$, we are led to the estimate
\begin{eqnarray}
\|K_{4}\|_{L^p}&\leq& C\sum_{k\geq
q+1}\|\dot{\Delta}_{k}f\|_{L^{p_{2}}}\|g\|_{L^{p_{1}}}\nonumber\\
&\leq& C
c_{q}2^{-q(s+1)}\|f\|_{\dot{B}^{s+1}_{p_2,r}}\|g\|_{L^{p_{1}}}.\label{R-E64}
\end{eqnarray}

For the estimate of $K_{5}$, we recall that
\begin{eqnarray*}
K_{5}=-\sum_{|q-k|\leq4}\dot{\Delta}_{q}(S_{k-1}g\dot{\Delta}_{k}f),
\end{eqnarray*}
furthermore, we get
\begin{eqnarray}
\|K_{5}\|_{L^p}&\leq& C\sum_{|q-k|\leq4}\|\dot{\Delta}_{k}f\|_{L^{p_{2}}}\|g\|_{L^{p_{1}}}\nonumber\\
&\leq& C
c_{q}2^{-q(s+1)}\|f\|_{\dot{B}^{s+1}_{p_2,r}}\|g\|_{L^{p_{1}}}.\label{R-E65}
\end{eqnarray}
Together with (\ref{R-E60})-(\ref{R-E65}), the inequality
(\ref{R-E58}) follows immediately.

In addition, the similar process enables us to obtain (\ref{R-E59}),
whereas the the time exponent $\theta$ behaves according to the
H\"{o}lder inequality.
\end{proof}

Following from the above proof,  it is not difficult to obtain the
commutator estimates in the inhomogeneous case.

\begin{cor}
For $s>-1$, $1\leq p\leq\infty$ and $ 1\leq r\leq\infty$, there is a
constant $C>0$ such that
\begin{eqnarray}
\|[f,\Delta_{q}]g\|_{L^{p}} \leq Cc_{q}2^{-q(s+1)}\Big(\|\nabla
f\|_{L^\infty}\|g\|_{B^{s}_{p,r}}+\|g\|_{L^{p_1}}\|f\|_{B^{s+1}_{p_2,r}}\Big)\label{R-E66}
\end{eqnarray}
and
\begin{eqnarray}
\|[f,\Delta_{q}]g\|_{L^{\theta}_{T}(L^{p})} \leq
Cc_{q}2^{-q(s+1)}\Big(\|\nabla
f\|_{L^{\theta_1}_{T}(L^\infty)}\|g\|_{\widetilde{L}^{\theta_2}_{T}(B^{s}_{p,r})}+\|g\|_{L^{\theta_3}_{T}(L^{p_1})}\|f\|_{\widetilde{L}^{\theta_4}_{T}(B^{s+1}_{p_2,r})}\Big),\label{R-E67}
\end{eqnarray}
where $1/p=1/p_{1}+1/p_{2}$,
$1/\theta=1/\theta_{1}+1/\theta_{2}=1/\theta_{3}+1/\theta_{4}$ and
$c_{q}$ denotes a sequence such that $\|(c_{q})\|_{\ell^{r}}\leq1$.
\end{cor}

\section*{Acknowledgments}
J. Xu is partially supported by the NSFC (11001127), Special
Foundation of China Postdoctoral Science Foundation (2012T50466),
China Postdoctoral Science Foundation (20110490134), Postdoctoral
Science Foundation of Jiangsu Province (1102057C) and the NUAA
Fundamental Research Funds (NS2013076). S. Kawashima is partially
supported by Grant-in-Aid for Scientific Research (A) 22244009.

\end{document}